\documentclass[11pt,a4paper,reqno]{amsart}

\usepackage{dynkin-diagrams}
\usepackage{kpfonts}
\usepackage[T1]{fontenc} 
\usepackage{amsmath}
\usepackage{amssymb} 
\usepackage{ulem}
\usepackage{ucs} 
\usepackage[utf8x]{inputenc}
\usepackage[]{graphicx} 
\usepackage{latexsym}
\usepackage{amsthm} 
\usepackage{stmaryrd}
\usepackage{verbatim}
\usepackage{pgf,tikz}
\usepackage{pgfkeys}
\usetikzlibrary{arrows}
\usetikzlibrary{backgrounds}
\usepackage{empheq}
\usepackage{hyperref}
\usepackage{enumitem}
\usepackage{tensor}
\usepackage{etoc}

\theoremstyle{definition}
\newtheorem{defi}{Definition}[section]

\newtheorem*{rmk}{Remark}

\theoremstyle{plain}

\newtheorem{theorem}[defi]{Theorem}
\newtheorem{question}[defi]{Question}
 
\newtheorem{lemme}[defi]{Lemma}
\newtheorem{lemma}[defi]{Lemma}
\newtheorem{prop}[defi]{Proposition}
\newtheorem{conj}[defi]{Conjecture}

\theoremstyle{remark}

\newcommand{\R}{\mathbb{R}}

\newcommand{\bbH}{\mathbb{H}}

\newcommand{\e}{\varepsilon}
\newcommand{\p}{\varphi}
\newcommand{\G}{\Gamma}

\newcommand{\cC}{\mathcal{C}}

\newcommand{\cR}{\mathcal{R}}
\newcommand{\cS}{\mathcal{S}}

\newcommand{\cU}{\mathcal{U}}

\newcommand{\cX}{\mathcal{X}}

\DeclareMathOperator{\Tr}{Tr}

\newcommand{\dS}{\mathrm{dS}}

\newcommand{\Diff}{\mathrm{Diff}}

\DeclareMathOperator{\dvol}{dvol}
\DeclareMathOperator{\Div}{div}

\DeclareMathOperator{\Hess}{Hess}

\newcommand{\Gr}{\mathrm{Gr}}

\newcommand{\intff}[2]{\mathopen{[}#1\,,#2\mathclose{]}}

\newcommand{\set}[2]{\left\lbrace #1 \,\middle|\, #2 \right\rbrace}

\newcommand{\pscal}[2]{\left\langle  #1 \middle| #2\right\rangle}

\renewcommand{\tilde}{\widetilde}

\newcommand{\norm}[1]{\left\Vert #1 \right\Vert}

\newcommand{\map}[4]{\left\lbrace \begin{array}{ccc} #1 & \to & #2 \\ #3 & \mapsto & #4 \end{array} \right.}

\newcommand{\Cauchy}{\mathfrak{C}}

\title[Geometry of the space of Cauchy hypersurfaces]{The Riemannian geometry of the space of compact spacelike Cauchy hypersurfaces} 
\author{Daniel Monclair}
\thanks{The author was partially supported by the ANR JCJC grant GAPR (ANR-22-CE40-0001, Geometry and Analysis in the Pseudo-Riemannian setting).}

\begin{document}

\normalem

\begin{abstract}
We study the geometry of a weak Riemannian metric on the infinite dimensional manifold of compact spacelike Cauchy hypersurfaces in a globally hyperbolic spacetime. We show that the geodesic distance (i.e. the infimum of lengths of paths between two given points) is positive, and that the sectional curvature is well defined and non-positive.
\end{abstract}

\maketitle

\section{Introduction}

A globally hyperbolic spacetime $(M,g)$ splits as a product $S\times \R$ in which the metric reads as $-\beta dt^2+g_t$ where $\beta:S\times\R\to\R$ is a smooth positive function and $(g_t)_{t\in\R}$ is a family of Riemannian metrics on the submanifold $S\subset M$, called a Cauchy hypersurface \cite{geroch,BS_surface_lisse,BS_splitting}. However, there is no canonical choice of a Cauchy hypersurface. In some cases, there are some preferred choices (e.g. constant mean curvature hypersurfaces, or level sets of the cosmological time), but they do not exist in an arbitrary globally hyperbolic spacetime.  In this paper we propose to study the space of all spacelike Cauchy hypersurfaces $\Cauchy(M,g)$ in a spatially compact spacetime, i.e. a globally hyperbolic spacetime whose Cauchy hypersurfaces are compact. We hope that this can eventually serve as an alternative to the choice of a special Cauchy hypersurface.

The space $\Cauchy(M,g)$ inherits the structure of an infinite dimensional manifold from a more general construction on the space of a submanifolds of a given manifold \cite{convenient,hamilton_nash_moser}, and  the tangent space $T_S\Cauchy(M,g)$ at $S\in \Cauchy(M,g)$ can be identified with $\cC^\infty(S)$. This allows us to define a Riemannian metric on $\Cauchy(M,g)$ by considering the inner product in $L^2(S)$, which we refer to as the $L^2$-metric on $\Cauchy(M,g)$. Our main focus will be the Riemannian geometry of the $L^2$-metric on $\Cauchy(M,g)$.

 One has to be careful in the study of infinite dimensional Riemannian manifolds, as many of the classic results fail in this generalization. For example, one can still define the geodesic distance between two points as the infimum of lengths of curves, but it may fail to be a distance because it could vanish. This is known to be the case when studying a similar $L^2$-metric on the space of submanifolds of a Riemannian manifold \cite{vanishing_distance}. In our case, we do get a distance.

\begin{theorem} \label{mainthm_distance}
Let $(M,g)$ be a spatially compact spacetime. The geodesic distance of the $L^2$-metric on $\Cauchy(M,g)$ is a positive distance.
\end{theorem}

Note however that this distance will not define the manifold topology on $\Cauchy(M,g)$. One of the explanations of the vanishing of the geodesic distance in the Riemannian case by Michor and Mumford \cite{vanishing_distance} is that the $L^2$-metric they consider has unbounded positive sectional curvature, so one can construct arbitrarily short paths by going through arbitrarily strongly curved regions. This strong curvature corresponds to rapidly oscillating deformations of submanifolds, and an equivalent construction in our Lorentzian setting is not possible because such deformation of spacelike submanifolds would no longer be spacelike.

In the Lorentzian case, one should not expect positive but rather negative curvature, as suggested by the example of de Sitter spacetime $\dS^n$. Considering the standard bilinear form $\pscal{\cdot}{\cdot}_{n,1}$ of signature $(n,1)$ on $\R^{n+1}$, one can define the   de Sitter spacetime
 \[\dS^n=\set{x\in \R^{n,1}}{\pscal{x}{x}_{n+1}=1}\]
 
  as well as the real hyperbolic space
  \[\bbH^n=\set{x\in\R^{n+1}}{\pscal{x}{x}_{n,1}=-1,~x_0>0}.\]

There is a natural embedding $\bbH^n\to\Cauchy(\dS^n)$ defined by $x\mapsto x^\perp\cap\dS^n$ which corresponds to  totally geodesic Cauchy hypersurfaces. This embedding is not only isometric, but actually totally geodesic, thus implying that there is some negative curvature in $\Cauchy(\dS^n)$. We will prove that the sectional curvature is always non positive.

\begin{theorem} \label{mainthm_curvature}
Let $(M,g)$ be a spatially compact spacetime. The sectional curvature of the $L^2$-metric on $\Cauchy(M,g)$ is non positive. More precisely, for $S\in\Cauchy(M,g)$ and an orthonormal pair  $u,v\in \cC^\infty(S)=T_S\Cauchy(M,g)$, the sectional curvature of the 2-plane spanned by $u$ and $v$ is:
\[ -\frac{1}{2}\int_S \norm{u\vec\nabla v-v\vec\nabla u}^2\dvol_S\]
where the gradient $\vec\nabla$, the norm $\norm{\cdot}$ and the measure $\dvol_S$ are considered with respect to the induced Riemannian metric on $S$.

\end{theorem}

Before we can compute any curvature, we need to show that it is well defined, as infinite dimensional Riemannian manifolds may fail to have a Levi-Civita connection. Although it would be possible to adapt the computations in \cite{vanishing_distance} to our case, we propose a completely different approach using orthogonal splittings of globally hyperbolic spacetimes. In particular, we will see that the abstract setting of infinite dimensional manifolds is not so important as we have a nice global chart.

\vspace{\baselineskip}

 Many elementary questions remain to be solved for this $L^2$-metric. One worth mentioning, even though we do not treat it in this paper, is geodesics. The geodesic equation of an infinite dimensional Riemannian manifold is a partial differential equation that does not always have solutions. The nature of the equation depends strongly on the metric, and for the $L^2$-metric on $\Cauchy(M,g)$ it seems to satisfy some very weak form of ellipticity, but not enough to fit in a  theory that guarantees (local) existence of solutions. Theorem \ref{mainthm_curvature} encourages us to believe that geodesics should be unique, as non positive curvature can be translated as convexity of the energy functional.

 \begin{conj}
 Given two spacelike Cauchy hypersurface $S_0,S_1\in \Cauchy(M,g)$, there is a unique geodesic path $c:\intff{0}{1}\to \Cauchy(M,g)$ such that $c(0)=S_0$ and $c(1)=S_1$. Moreover, its length minimizes the lengths of paths from $S_0$ to $S_1$, and it is the only only one to do so.
 \end{conj}

In a finite dimensional Riemannian manifold, non-positive sectional curvature has a strong consequence on the Riemannian distance, namely the $\rm{CAT}(0)$-inequality \cite[p.173]{bridson_haefliger}. This relationship between the curvature and the Riemannian distance is best understood through the behaviour of geodesics. This is still out of reach in our setting because we not only need to have the existence of geodesics in the Riemannian sense, but also prove that they are geodesics in the metric sense.  One can still ask if a metric space is $\rm{CAT}(0)$ without using geodesics thanks to the 4-point condition \cite[p.164]{bridson_haefliger}, which requires metric completeness in order to be applied.

\begin{question}
Is  $\Cauchy(M,g)$ (or its metric completion) a $\rm{CAT}(0)$-space?
\end{question}

Of course the geodesic distance associated to the $L^2$-metric is not complete. For once, if we start with a spatially compact spacetime $(M,g)$ and consider a spatially compact open domain $U\subset M$, then $\Cauchy(U,g)$ is an open set of $\Cauchy(M,g)$ with the same Riemannian metric. But the lack of completeness has other reasons: one can easily construct paths of finite length converging to some lightlike hypersurface. The metric completion of $\Cauchy(M,g)$ should also include some hypersurfaces of lower regularity.

One of the motivations for this question is Lorentzian dynamics. Indeed, if a group $G$ acts isometrically on a spatially compact spacetime $(M,g)$, then it also acts on $\Cauchy(M,g)$ and its metric completion $X$. If the action of $G$ on $M$ is non proper, it cannot preserve a spacelike Cauchy hypersurface. However   $G$ also acts on the visual boundary $\partial_\infty X$, and one can look for fixed points there. One can expect such a fixed point to be some kind of lightlike hypersurface, and if preserved by enough isometries it should be totally geodesic. Totally geodesic lightlike hypersurfaces play a central role in the work of Zeghib \cite{Z99a,Z99b} on the classification of non compact connected Lie groups acting on compact Lorentzian manifolds. A classification of groups acting non properly on two dimensional spatially compact spacetimes was obtained in \cite{isometries_lorentz_surfaces}. The action on the space of spacelike Cauchy hypersurfaces could potentially help to extend this classification to higher dimensions.

\vspace{\baselineskip}

Compactness of Cauchy hypersurfaces may appear as a technical assumption at a first glance, mainly being required so that the $L^2$ metric is well defined. This could be bypassed by using the convenient setting of \cite{convenient}, as for any globally hyperbolic spacetime $(M,g)$ the space of spacelike Cauchy hypersurfaces $\Cauchy(M,g)$ carries the structure of an infinite dimensional manifold for which the tangent space $T_S\Cauchy(M,g)$ identifies with the space $\cC_c^\infty(S)$ of compactly supported smooth functions on $S$. The problem in this case is that the topology obtained on $\Cauchy(M,g)$ is not connected, so the Riemannian distance is infinite.

Another reason why this formalism is probably not adapted to the case of non compact Cauchy hypersurfaces is the example of totally geodesic Cauchy hypersurfaces in Minkowski spacetime. Indeed, the natural geometry on this space, called co-Minkowski space, is a degenerate (sub-Riemannian) metric, see e.g. \cite{co_minkowski}.

\subsection*{Outline}

The manifold structure and the Riemannian metric on the space $\Cauchy(M,g)$ of spacelike Cauchy hypersurfaces are presented in section \ref{sec: manifold of Cauchy hypersurfaces}, starting with all the necessary notions from Lorentzian geometry and concluding with a proof of Theorem \ref{mainthm_distance}. We then compute the Levi-Civita connection of $\Cauchy(M,g)$ in section \ref{sec: Levi-Civita connection}, thus proving its existence, and finally compute the curvature in section \ref{sec: curvature}.

\section{The manifold of Cauchy hypersurfaces} \label{sec: manifold of Cauchy hypersurfaces}

\subsection{Basic Lorentzian vocabulary}
In this paper a Lorentzian manifold $(M,g)$ is  $\cC^\infty$ and has  signature $(-,+,\dots,+)$.  A non zero  tangent vector $v\in T_xM$ is called \textbf{causal} if $g_x(v,v)\leq 0$.  A \textbf{time orientation} of $(M,g)$ is  a continuous choice of a connected component of the set of  causal tangent vectors at a point\footnote{Formally, it is an equivalence class of non vanishing causal vector fields under the identification $T_1\sim T_2$ if $g_x(T_1(x),T_2(x))\leq 0$ for all $x\in M$.}. We will call a time oriented Lorentzian manifold a \textbf{spacetime}. It allows us to define \textbf{future} (resp. \textbf{past}) \textbf{directed} causal vectors as vectors in the chosen component (resp. in the other component).

  A \textbf{future  directed curve} is a smooth curve $c:I\to M$ such that $\dot c(t)$ is future directed for all $t\in I$. It is called \textbf{inextensible} if it is not the reparametrisation of a future directed curve defined on a larger interval.

\vspace{\baselineskip}

  A \textbf{Cauchy hypersurface} is a topological hypersurface $S\subset M$ such that any inextensible future directed curve intersects $S$ at exactly one point.  A spacetime  is called \textbf{globally hyperbolic} if it possesses a Cauchy hypersurface, and \textbf{spatially compact} if it possesses a compact Cauchy hypersurface.

 A  \textbf{spacelike hypersurface} is a smooth hypersurface $S\subset M$ such that induced metric $g\vert_S$ has Riemannian signature. The space $\Cauchy(M,g)$ under study is the set of Cauchy hypersurfaces of $M$ that are spacelike.
 
 \vspace{\baselineskip}

 The two notions of a Cauchy hypersurface and a spacelike hypersurface can look very similar at a first glance, but they are not the same.  Not all Cauchy hypersurfaces are  spacelike: they may not be smooth, but they can also be smooth with some points at which the induced metric is degenerate. In an arbitrary Lorentzian manifold, spacelike hypersurfaces are not necessarily Cauchy hypersurfaces (they can exist in non globally hyperbolic Lorentzian manifolds).
 
 In the globally hyperbolic case however, it was proved in \cite{BS_splitting} that every compact spacelike hypersurface with boundary is included in a spacelike Cauchy hypersurface. Assuming spatial compactness, we get the following result.

% Given a point $x\in M$, we define its future $J^+(x)$ (resp. its past $J^-(x)$) to be the set of endpoints of future (resp. past) curves starting at $x$. \\

%\begin{defi} \label{def - spatially compact}
%Let $(M,g)$ be a Lorentzian Lorentzian manifold. A \textbf{Cauchy hypersurface} is a topological hypersurface $S\subset M$ such that every inextensible causal curve intersects $S$ exactly once. We say that $(M,g)$ is \textbf{spatially compact} if it admits a compact Cauchy hypersurface. 
%\end{defi}

\begin{theorem}[{\cite[Theorem 1.1]{BS_splitting}}] \label{thm - spacelike hypersurfaces are Cauchy}
Let $(M,g)$ be a spatially compact spacetime. Any smooth closed spacelike hypersurface of $M$ is a Cauchy hypersurface.
\end{theorem}

This means that $\Cauchy(M,g)$ is just the set of closed spacelike hypersurfaces, in particular it is open in the space of all embedded hypersurfaces in $M$, and thus inherits a strucutre of an infinite dimensional manifold \cite{hamilton_nash_moser,convenient}. Note that $\Cauchy(M,g)$ is non empty thanks to \cite[Theorem 1.1]{BS_surface_lisse}.

\subsection{Orthogonal splittings}

Given a spacelike Cauchy hypersurface $\Sigma$ in a Lorentzian manifold $(M,g)$, it is easy to see that $M$ is diffeomorphic to the product $\Sigma\times \R$: the flow of a complete everywhere timelike vector field will provide such an identification. However, the metric thus obtained on $\Sigma\times \R$ is so far from being a product metric that it is not very helpful for computations. We will use a stronger result showing that such a splitting can be chosen to be orthogonal, i.e. all slices $\Sigma\times\{ t\}$ are orthogonal to $\frac{\partial\,}{\partial t}$.

\begin{theorem}[{\cite[Theorem 1.2]{BS_splitting}}] \label{thm - orthogonal splitting}
Let $(M,g)$ be a Lorentzian manifold, and $\Sigma\subset M$ a spacelike Cauchy hypersurface. There is a diffeomorphism $\p:\Sigma\times \R\to M$ with the following properties:
\begin{itemize}
\item $\p(x,0)=x$ for all $x\in \Sigma$
\item The pullback metric $\p^*g$ decomposes as $-\beta dt^2 +g_t$ where $\beta:\Sigma\times \R\to \R$ is a smooth positive function and $g_t$ is a smooth  family of Riemannian metrics on $\Sigma$.
\end{itemize}
\end{theorem}

It is important to keep in mind that the function $\beta$, called the \emph{lapse} function, is a function of both variables $\beta:\Sigma\times \R\to \R_{>0}$. So even when $\Sigma$ is compact, it may not be bounded. This will not be relevant for curvature computations, but we will use orthogonal splittings with specific bounds on the lapse function in order to prove the positivity of the Riemannian distance.

 It was shown by M\"uller and S\'anchez \cite[Theorem 1.2]{lorentzian_nash} that one can always choose the lapse function to be bounded from above. They used this result to show that any globally hyperbolic spacetime can be isometrically embedded in the Minkowski space of large enough dimension, which can be seen as a Lorentzian analogue of the Nash Embedding Theorem in the globally hyperbolic case. 
 
 For our applications, we will need a bound from below.

\begin{prop} \label{prop - bounded lapse function}
Let $(M,g)$ be a Lorentzian manifold, $\Sigma\subset M$ a compact spacelike Cauchy hypersurface and $h$ a Riemannian metric on $\Sigma$. There is a diffeomorphism $\p:S\times I\to M$ where $I\subset\R$ is an open interval  with the following properties:
\begin{itemize}
\item $\p(x,0)=x$ for all $x\in \Sigma$
\item The pullback metric $\p^*g$ decomposes as $-\beta dt^2 +g_t$ where $\beta:\Sigma\times I\to \R$ is a smooth positive function and $g_t$ is a smooth  family of Riemannian metrics on $\Sigma$.
\item $\beta(x,t) \sqrt{ \det\left(h_x^{-1}(g_t)_x\right)}\geq 1$ for all $x\in \Sigma$ and $t\in I$.
\end{itemize}
\end{prop}

\begin{proof}
Consider the decomposition given by Theorem \ref{thm - orthogonal splitting}, and consider the function $m:\R\to \R$ defined by:
 \[\forall t\in\R\quad m(t)=\min_{x\in \Sigma}\left(\beta(x,t) \sqrt{ \det\left(h_x^{-1}(g_t)_x\right)}\right)^{\frac{1}{2}}\]
 
  Note that this function is positive and continuous, but not necessarily smooth. We can however consider a smooth function $F:\R\to \R$ such that $F(t)\geq \max( 1,1/m(t) )$ for all $t\in \R$. Now let $f:I\to \R$ be the maximal solution of the ordinary differential equation $f'=F(f)$ with inital data $f(0)=0$. The condition $F\geq 1$ guarantees that $f$ is a diffeomorphism from $I$ to $\R$.\\
 When pulling back the metric on $\Sigma\times \R$ to $\Sigma\times I$ by the diffeomorphism $(x,t)\mapsto (x,f(t))$, the metric reads as $-\tilde \beta(x,t)dt^2+\tilde g_t$ where the new lapse function $\tilde \beta$ is given by $\tilde \beta(x,t)=\beta(x,f(t))f'(t)^2$ and the family of Riemannian metrics $\tilde g_t$ on $\Sigma$ is given by $(\tilde g_t)_x=(g_{f(t)})_x$.
 
 \begin{align*}
 \tilde\beta(x,t) \sqrt{ \det\left(h_x^{-1}(\tilde g_t)_x\right)} & = \beta(x,f(t))f'(t)^2 \sqrt{ \det\left(h_x^{-1}(g_{f(t)})_x\right)}\\
 &=F( f(t))^2\beta(x,f(t)) \sqrt{ \det\left(h_x^{-1}(g_{f(t)})_x\right)}\\
  &\geq\frac{1}{m(f(t))^2}\beta(x,f(t)) \sqrt{ \det\left(h_x^{-1}(g_{f(t)})_x\right)}\\
 &\geq 1
 \end{align*}
 
\end{proof}

\subsection{The intrinsic manifold structure of \texorpdfstring{$\Cauchy(M,g)$}{C(M,g)}}
Given a manifold $M$ and a compact manifold $\Sigma$, the space $\cS(\Sigma,M)$ of all submanifolds of $M$ that are diffeomorphic to $\Sigma$ can be given the structure of an infinite dimensional manifold. This is described in \cite[Theorem 44.1]{convenient} and \cite[Example 4.1.7]{hamilton_nash_moser}. It can actually be viewed as the base of a principal bundle by considering the quotient $\cS(\Sigma,M)={\rm Emb}(\Sigma,M)/\Diff(\Sigma)$, where ${\rm Emb}(\Sigma,M)$ is the submanifold of $\cC^\infty(\Sigma,M)$ consisting of embeddings, and $\Diff(\Sigma)$ acts on ${\rm Emb}(\Sigma,M)$ by precomposition.

A deformation of a submanifold consists in a normal vector field on the submanifold, i.e. the tangent space at $S\in \cS(\Sigma,M)$ can be naturally identified with the space of sections $\G(\nu S)$ where $\nu S$ is the normal bundle defined by $\nu_x S=T_xM/T_x S$ for $x\in S$ \cite[Example  4.5.5]{hamilton_nash_moser}.

\vspace{\baselineskip}

Now let us come back to our setting where $(M,g)$ is a spatially compact spacetime, and $\Cauchy(M,g)$ is the set of spacelike Cauchy hypersurfaces in $(M,g)$. Fix  a smooth Cauchy hypersurface $\Sigma\subset M$. Since all Cauchy hypersurfaces are diffeomorphic to each other, we find that $\Cauchy(M,g)\subset \cS(\Sigma,M)$. It happens that $\Cauchy(M,g)$ is open in $\cS(\Sigma,M)$, but this is not obvious from the definition of Cauchy hypersurfaces, its comes from the fact that all spacelike closed hypersurfaces are Cauchy hypersurfaces \cite[Theorem 1.1]{BS_splitting}.

%\begin{theorem}[{\cite[Theorem 1.1]{BS_splitting}}]
%Let $(M,g)$ be a spatially compact spacetime. Any  spacelike embedded closed hypersurface is a Cauchy hypersurface.
%\end{theorem}

Spacelike being  an open condition in $\cS(\Sigma,M)$, we find that $\Cauchy(M,g)$ is open in $\cS(\Sigma,M)$. The tangent space  $T_S\Cauchy(M,g)$ at $S\in\Cauchy(M,g)$ is the space of sections of the normal bundle of $S$. Since $S\in \Cauchy(M,g)$ is a spacelike hypersurface, there is a unique normal vector field $n_S:S\to TM$ such that $g(n_S,n_S)=-1$ and $n_S$ is future directed. Any section of the normal bundle can be represented as $un_S$ for some function $u\in \cC^\infty(S)$, and we will use this identification between $T_S\Cauchy(M,g)=\G(\nu S)$ and $\cC^\infty(S)$ to describe tangent vectors.

%{\color{red}
%Setting: that of Kriegl-Michor "The convenient setting of global analysis". Basically the setting in which everything works, except that smooth functions may not be continuous. See also Hamilton's paper on the Nash--Moser Theorem \cite{hamilton_nash_moser}. We find that $\Cauchy$ has a manifold structure, and that there is a natural identification between $T_S\Cauchy$ and $\G(\nu S)$, hence with $\cC^\infty(S)$.
%}

\subsection{Coordinates on \texorpdfstring{$\Cauchy(M,g)$}{C(M,g)}}

Let $(M,g)$ be a spatially compact spacetime, $\Cauchy(M,g)$ the space of spacelike Cauchy hypersurfaces and $\Sigma\in\Cauchy$. Following Theorem \ref{thm - orthogonal splitting}, we may assume that $M=\Sigma\times \R$, and that the metric $g$ on $M$ splits as $g_{(x,t)}=-\beta(x,t)dt^2+(g_t)_x$  for all $(x,t)\in\Sigma\times\R=M$ where $\beta$ is a smooth positive function on $\Sigma\times \R$ and $t\mapsto g_t$ is a smooth family of Riemannian metrics on $\Sigma$.

 The abstract setting of infinite dimensional manifolds is actually  not quite required for our study since we can find a (non canonical) global chart for $\Cauchy(M,g)$.

\begin{lemme} \label{lem - global chart for the space of Cauchy hypersurfaces}
The graph map
\[ \Gr:\map{\cU}{\Cauchy}{f}{\set{(x,f(x))}{x\in \Sigma}}\]
is a diffeomorphism from the open set 
\[ \cU=\set{f\in\cC^\infty(\Sigma)}{~\forall (x,z)\in T\Sigma\quad \beta(x,f(x))\,d_xf(z)^2<(g_{f(x)})_x(z,z)} \] to the space of Cauchy hypersurfaces $\Cauchy(M,g)$.
\end{lemme}

\begin{proof}
The condition defining $f\in \cU$ is exactly that $\Gr(f)$ is spacelike, so the fact that $f$ is a bijection follows from \cite[Theorem 1.1]{BS_splitting}. The smoothness of $\Gr$ and its inverse follow directly from the definition of the manifold structure as the quotient ${\rm Emb(\Sigma,M)}/\Diff(\Sigma)$ considered in \cite[Theorem 44.1]{convenient} and \cite[Example 4.1.7]{hamilton_nash_moser}.
\end{proof}

 In order to compute the Riemannian metric in this chart, we need to understand the differential of the map $\Gr$. Recall that the tangent space  $T_S\Cauchy(M,g)$ is the space of sections of the normal bundle of $S$, and is identified with $\cC^\infty(S)$ through the use of the  future directed unit normal vector field $n_S:S\to TM$.
 
\begin{lemma}\label{lem - differential graph map} Let $f\in\cU$, $u\in \cC^\infty(\Sigma)$ and $x\in \Sigma$. The differential of the map $\Gr:\cU\to\Cauchy$ is given by:
 \[ D_f\Gr(u)(x,f(x))=\left(\frac{\beta(x,f(x))}{1-\beta(x,f(x))\,\Vert d_xf\Vert_{ g_f}^2}\right)^{\frac{1}{2}} u(x)\]
 where $g_f$ is the Riemannian metric on $\Sigma$ defined by $(g_f)_x=(g_{f(x)})_x$ for $x\in\Sigma$.
\end{lemma}
\begin{rmk}
Throughout this paper we will use the notation $D$ for the differential of a map defined on an infinite dimension space (usually on $\cU$), and the small case letter $d$ for differential of maps on finite dimensional manifolds.
\end{rmk}
\begin{proof}
We first compute the future directed unit normal vector field $n_{\Gr(f)}$ of $\Gr(f)$. Let $x\in \Sigma$, and write $n_{\Gr(f)}(x,f(x))= (v,\lambda)$. The fact that $n_{\Gr(f)}$ is future directed translates as $\lambda>0$. Since $T_{(x,f(x))}\Gr(f)=\Gr(d_xf)$, the vector $n_{\Gr(f)}(x,f(x))$ is orthogonal to $(z,d_xf(z))$ for any $z\in T_x\Sigma$, i.e. 
\[\forall z\in T_x\Sigma\quad (g_{f(x)})_x(z,v) = \beta(x,f(x)) \lambda d_xf(z)\]
It follows that :
\[ v=\lambda \beta(x,f(x))\vec\nabla_f f(x)\]
Here $\vec\nabla_f$ represents the gradients for the Riemannian metric $g_f$ on $\Sigma$.
Since $n_{\Gr(f)}(x,f(x))$ is a unit  timelike vector, we find:
\[ \lambda^2 \left[\beta(x,f(x))^2\Vert d_x f\Vert_{ g_f}^2 -\beta(x,f(x)) \right] =-1\]
Finally:
\[ n_{\Gr(f)}(x,f(x))=\left(\beta(x,f(x))-\beta(x,f(x))^2\Vert d_x f\Vert_{ g_f}^2\right)^{-\frac{1}{2}} \left( \beta(x,f(x))\vec\nabla_f f(x), 1\right)\]
Now $D_f\Gr(u)(x,f(x))$ is obtained by projecting orthogonally $(0,u(x))$ on the line spanned by $n_{\Gr(f)}(x,f(x))$. Note that since $n_{\Gr(f)}$ is timelike, there is a sign difference compared to the orthogonal projection in the spacelike case.
\begin{align*}
D_f\Gr(u)(x,f(x)) &= -g_{(x,f(x))}\left(\rule{0cm}{0.4cm}(0,u(x)),n_{\Gr(f)}(x,f(x))\right)\\
&= \left(\frac{\beta(x,f(x))}{1-\beta(x,f(x))\Vert d_xf\Vert_{ g_f}^2}\right)^{\frac{1}{2}} u(x) 
\end{align*}

\end{proof}

\subsection{The Riemannian metric on \texorpdfstring{$\Cauchy(M,g)$}{C(M,g)}}

We now consider the Riemannian metric on $\Cauchy(M,g)$ given by the $L^2$-product on tangent spaces through the identification $T_S\Cauchy(M,g)=\cC^\infty(S)$. This is a weak Riemannian metric, i.e. a smooth positive definite tensor of type $(2,0)$. It is called weak because the topology on the tangent space induced by its norm (in our case the $L^2$ topology) is not the topology of the tangent space (in our case the $\cC^\infty$ topology).

\begin{lemma} \label{lem - metric in coordinates} The pull-back metric $G$ on $\cU$ by the graph function $\Gr$ is given by:
\[\forall f\in\cU~\forall u,v\in \cC^\infty(\Sigma)~~ G_f(u,v)=\int_\Sigma \frac{\beta(x,f(x))}{\sqrt{1-\beta(x,f(x))\Vert d_xf\Vert^2_{g^f}}} u(x)v(x) \dvol_{g_f}(x)\]
where $g_f$ is the Riemannian metric on $\Sigma$ defined by $(g_f)_x=(g_{f(x)})_x$ for $x\in \Sigma$.
\end{lemma}

\begin{proof}
By definition, we have:

\[G_f(u,v)=\int_{\Gr(f)}D_f\Gr(u) D_f\Gr(v)\dvol_g
\]
The integrand comes from Lemma \ref{lem - differential graph map}:
\[ D_f\Gr(u)D_f\Gr(v) =  \frac{\beta(x,f(x))}{1-\beta(x,f(x))\Vert  d_xf\Vert_{ g_f}^2}u(x)v(x)
\]
We now must compute the pull-back of $\dvol_g$ on $\Sigma$ by the map $x\mapsto (x,f(x))$. The Lorentzian metric $g$ pulls back to a Riemannian metric $h$ on $\Sigma$  given by:
\begin{equation} \label{eqn - pullback metric on the Cauchy hypersurface} h_x(z,z)=-\beta(x,f(x))d_xf(z)^2 + (g_f)_x(z,z)
\end{equation}
The Riemannian metric on $\cU$ written as an integral over $\Sigma$ is:
\[ G_f(u,v)= \int_\Sigma \frac{\beta(x,f(x))}{1-\beta(x,f(x))\Vert  d_xf\Vert_{ g_f}^2}u(x)v(x) \dvol_h(x) \]

To compute $\dvol_h(x)=\sqrt{\det\left(\left((g_f)_x\right)^{-1}h_x\right)}\dvol_{g_f}(x)$, we consider an orthonormal basis $(e_1,\dots,e_d)$ of $T_x\Sigma$ for the metric $(g_f)_x$ where $e_1=\frac{\vec\nabla_f f}{\Vert \vec\nabla_f f\Vert}$ (unless $d_xf=0$, in which case $h_x= (g_f)_x$ and $\det\left(\left((g_f)_x\right)^{-1}h_x\right)=1$). Here $\vec\nabla_f$ is once again the gradient for the Riemannian metric $g_f$.

 We just have to compute the matrix $(h_x(e_i,e_j))_{1\leq i,j\leq d}$, which is easy thanks to \eqref{eqn - pullback metric on the Cauchy hypersurface}.

\[h_x(e_i,e_j)=\left\lbrace \begin{array}{ccl} 1-\beta(x,f(x)) \Vert d_xf\Vert^2_{(g_f)_x} & \textrm{ if } & i=j=1\\
1 & \textrm{ if } & i=j\geq 2 \\ 0 & \textrm{ if } & i\neq j \end{array}\right.
\]

 This gives:
\[\det\left(\left((g_f)_x\right)^{-1}h_x\right)=1-\beta(x,f(x))\Vert d_xf\Vert^2_{(g_f)_x}
\]

\end{proof}

\subsection{Positive distance} The definition of the Riemannian distance is the same as for finite dimensional manifolds.

\begin{defi} \label{def - riemannian distance}
Let $S,S'\in \Cauchy(M,g)$. The \emph{Riemannian distance} $d_{\Cauchy(M,g)}(S,S')$ is the infimum of lengths of piecewise smooth paths joining $S$ and $S'$.
\end{defi}

The triangle inequality for the Riemannian distance can be shown as in the finite dimensional case by considering concatenations of paths. The major drawback of this infinite dimensional setting is that it could potentially vanish for some pairs of distinct points. As mentionned in the introduction, this actually happens when considering submanifolds of Riemannian manifolds \cite{vanishing_distance}.

We will now prove that $d_{\Cauchy(M,g)}$  is a positive distance on $\Cauchy(M,g)$, but the topology induced by this distance is not the manifold topology of $\Cauchy(M,g)$, it has more open sets.

\begin{proof}[Proof of Theorem \ref{mainthm_distance}]
Let $h$ be a Riemannian metric on $\Sigma$. According to Proposition \ref{prop - bounded lapse function}, it is possible to assume that $M=\Sigma\times I$ where $I\subset \R$ is an open interval and the metric $g$ splits as $g=-\beta dt^2+g_t$ where $\beta:\Sigma\times I\to \R$ is a smooth function  and $g_t$ is a smooth one-parameter family of Riemannian metrics on $\Sigma$ satisfying the following inequality:
\[\forall (x,t)\in\Sigma\times I\quad  \beta(x,t) \sqrt{ \det\left(h_x^{-1}(g_t)_x\right)}\geq 1 \]

 Lemma \ref{lem - global chart for the space of Cauchy hypersurfaces} still holds if we replace $\R$ with $I$ when needed in the expression of $\cU\subset\cC^\infty(\Sigma,I)$, and the expression of the metric in coordinates from  Lemma \ref{lem - metric in coordinates} associated to the previous inequality  yields
\[\forall f\in\cU~\forall u\in\cC^\infty(\Sigma)\quad G_f(u,u)\geq \int_\Sigma u(x)^2 \dvol_{h}(x) \]
Now consider a smooth path $c:\intff{0}{1}\to \cU$ (and denote by $s$ the variable in $\intff{0}{1}$).  The length of $c$ can also be bounded from below:

\[L_{\Cauchy(M,g)}(c)\geq \int_0^1 \left( \int_\Sigma \left(\frac{\partial c}{\partial s}\right)^2 \dvol_{h}(x) \right)^{\frac{1}{2}}  \] 

We now recognize the length of a curve in the Hilbert space $L^2(\Sigma,\dvol_{h})$, so we get the usual lower bound by the norm:

\[ L_{\Cauchy(M,g)}(c) \geq  \Vert c(0)-c(1)\Vert_{L^2(\Sigma,\dvol_h)} \]
We now fix the endpoints $S_0=\Gr(f_0),S_1=\Gr(f_1)$ of $c$, so we find:
\[ d_{\Cauchy(M,g)}( S_0,S_1 ) \geq  \Vert f_0-f_1\Vert_{L^2(\Sigma,\dvol_h)} \]

\end{proof}

\section{The Levi-Civita connection of \texorpdfstring{$\Cauchy(M,g)$}{C(M,g)}} \label{sec: Levi-Civita connection}

For the rest of this paper, we consider a product manifold $M=\Sigma \times \R$ where $\Sigma$ is a compact manifold, and a Lorentzian metric $g$ on $M$ that split as $-g_{(x,t)}=\beta(x,t)dt^2+(g_t)_x$ where $\beta:\Sigma\times\R\to\R$ is a smooth positive function and $t\mapsto g_t$ is a smooth path of Riemannian metrics on $\Sigma$. As in Lemma \ref{lem - global chart for the space of Cauchy hypersurfaces} we consider the open set \[\cU=\{ f\in\cC^\infty(\Sigma)\, \vert\, \forall x\in\Sigma\, \forall z\in T_x\Sigma \quad\beta(x,f(x))\,d_xf(z)^2< (g_{f(x)})_x(z,z)\}.\]
and the weak Riemannian metric on $\cU$  defined by:
\[ \boxed{G_f(u,v)=\int_\Sigma \frac{\beta(x,f(x))}{(1-\beta(x,f(x))\Vert d_xf\Vert_{(g_f)_x}^2)^{\frac{1}{2}}}u(x)v(x) \dvol_{g_f}(x)}\]
for $f\in \cU$ and $u,v\in\cC^\infty(\Sigma)$, where $g_f$ is the Riemannian metric on $\Sigma$ given by $(g_f)_x=(g_{f(x)})_x$.

\subsection{Notations} Given $f\in \cU$, we will always denote by $g_f$ the Riemannian metric on $\Sigma$ defined by $(g_f)_x=(g_{f(x)})_x$. Most Riemannian considerations on $\Sigma$ will be  with respect to $g_f$. For any function $u\in\cC^\infty(\Sigma)$, we will use the notation $\vec\nabla_fu$ for the gradient of $u$ with respect to the metric $g_f$, and for vector fields $X,Y\in\cX(\Sigma)$ we write $\nabla^f_XY$ for the Levi-Civita connection $\nabla^f$ of the Riemannian metric $g_f$.
 
 \vspace{\baselineskip}

 We also define $F(x,z,t)=(1-\beta(x,t)g_x^t(z,z))^{-\frac{1}{2}}$ for $(x,z)\in T\Sigma$ and $t\in\R$ such that $\beta(x,t)g_x^t(z,z)<1$. We will use the notations $\beta_f$ and $F_f$ for $\beta_f(x)=\beta(x,f(x))$ and $F_f(x)=F(x,\vec\nabla_f f(x),f(x))$. This simplifies the expression of the Riemannian metric:
\[ \boxed{G_f(u,v)=\int_\Sigma \beta_f F_f uv \dvol_{g_f}} \]

The first important computation will be the Levi-Civita connection of $G$. Since there is no existence theorem for weak Riemannian metrics, we have no other option but to compute it.

\subsection{Some differentials} The Levi-Civita connection requires a differentiation of $G_f(u,v)$ with respect to $f$. For this, we must consider the maps $\beta_\bullet:f\mapsto \beta_f$ and $F_\bullet:f\mapsto F_f$. We will also need to differentiate the volume element $\dvol_{g_\bullet}:f\mapsto\dvol_{g_f}$. These are all smooth functions of $f$ because they all depend on some finite order jet of $f$.
 
 The easiest functions to differentiate are those that only depend on the value of $f$ at a given point (i.e. its $0$-jet), and not on its derivatives. This is the case of the volume element.
 
\begin{lemme} \label{lem - differential volume element} The differential of the volume element is:
\[ D_f(\dvol_{g_\bullet})(u)=\frac{1}{2}\Tr\left( (g_f)^{-1}h_{f}\right)u\dvol_{g_f}\]

where  $h_t=\left.\frac{d}{ds}\right\vert_{s=t}g_s$, and $h_f$ is the tensor of type $(2,0)$ on $\Sigma$ defined by $(h_f)_x=(h_{f(x)})_x$.

\end{lemme}
\begin{proof}
In local coordinates $(x^1,\dots,x^d)$ on $\Sigma$, one has: \[\dvol_{g_t}(x)=\sqrt{\det(g_t(x))}dx^1\wedge\dots\wedge dx^d\]
 where $g_t(x)$ is the matrix of $(g_t)_x$ in these coordinates. Since it is invertible, we have:
 \[ \frac{d}{dt}\det(g_t(x))=\det(g_t(x))\Tr\left(g_t(x)^{-1}h_t(x)\right)\]
  We now get:
 \[ \frac{d}{dt}\sqrt{\det(g_t(x))}=\frac{1}{2}\det(g_t(x))\Tr\left(g_t(x)^{-1}h_t(x)\right)\left(\det(g_t(x))\right)^{-\frac{1}{2}}\]
 The formula follows directly.
\end{proof}

In order to differentiate $f\mapsto F_f$, we need to differentiate $f\mapsto \vec\nabla_f f$, which is not linear as the gradient $\vec\nabla_f$ depends on the metric $g_f$.
\begin{lemme} For any fixed $w\in \cC^\infty(\Sigma,\R)$, the differential of  $\vec\nabla_\bullet w:f\mapsto \vec\nabla_f w$ is given by:
\[ D_f(\vec\nabla_\bullet w)(u)=-u\widetilde{h_f}(\vec\nabla_f w)\]
where $\widetilde{h_t}$ is the $g_t$-self adjoint  endomorphism of $T\Sigma$  satisfying $g_t(\widetilde{h_t}(\cdot),\cdot)=h_t$.
\end{lemme}
\begin{proof}
We differentiate both signs of $dw(z)=g_f(\vec\nabla_f w,z)$ with respect to $f$ along the function $u$ and get 
\[ 0=g_f(D_f(\vec\nabla_\bullet w)(u),z)+uh_f(\vec\nabla_f w,z) \]
\end{proof}
\begin{lemme}\label{lem - differential gradient}  The differential of $\vec\nabla:f\mapsto \vec\nabla_f f$ is given by:
\[ D_f\vec\nabla(u)=\vec\nabla_f u-u\widetilde{h_f}(\vec\nabla_f f).\]
%where $\widetilde{h_t}$ is the $g_t$-self adjoint  endomorphism of $T\Sigma$  satisfying $g_t(\widetilde{h_t}(\cdot),\cdot)=h_t$.
\end{lemme}
\begin{proof}
By definition of the gradient, we have that: \[ \forall x\in\Sigma~\forall z\in T_x\Sigma\quad d_xf(z)=(g_f)_x(\vec\nabla_f f(x),z)\]
 the linearity and continuity of the map $f\mapsto d_xf$ leads us to:
\[ d_xu(z)=u(x)(h_f)_x(\vec\nabla_f f(x), z)+(g_f)_x(D_f\vec\nabla(u)(x),z)\]
Since $d_xu(z)=g_f(\vec\nabla_f u(x),z)$, we get the desired formula.
\end{proof}
We can now differentiate $f\mapsto F_f$.

\begin{prop} \label{prop - differential of F} The differential of $f\mapsto F^f$ is:
 \[D_fF_\bullet (u)= \beta_f F_f^3\left[g_f(\vec\nabla_f f,\vec\nabla_f u)+\xi(f)u\right]
\]
where 
\[\xi(f)=\frac{1}{2}\left(\frac{1}{\beta}\frac{\partial \beta}{\partial t}\right)_fg_f\left(\vec\nabla_f f,\vec\nabla_f f\right)-\frac{1}{2} h_f\left(\vec\nabla_f f,\vec\nabla_f f\right)\in \cC^\infty(\Sigma)\]
\end{prop}
\begin{proof}The differential of $\beta_\bullet$ is straightforward:
\[ D_f\beta_\bullet(u)= \left(\frac{\partial \beta}{\partial t}\right)_f u.\]
Consider $E_f=F_f^{-2}=1-\beta_fg_f\left(\vec\nabla_f f,\vec\nabla_f f\right)$. From Lemma \ref{lem - differential gradient}, we get:
\[ \begin{array}{rcl}
D_fE_\bullet(u)&= - g_f\left(\vec\nabla_f f,\vec\nabla_f f\right) D_f\beta_\bullet(u)&-2\beta_fg_f\left(\vec\nabla_f f, D_f\vec \nabla (u)\right)\\
&& -\beta_f h_f\left(\vec \nabla_f f,\vec \nabla_f f\right)u\\
&= -g_f\left(\vec\nabla_f f,\vec\nabla_f f\right) \left(\frac{\partial \beta}{\partial t}\right)_f u &- 2\beta_fg_f\left(\vec\nabla_f f, \vec\nabla_f u-u\widetilde{h_f}(\vec\nabla_f f)\right)\\
& & -\beta_f h_f\left(\vec\nabla_f f,\vec\nabla_f f\right)u\\
&=-g_f\left(\vec\nabla_f f,\vec\nabla_f f\right) \left(\frac{\partial \beta}{\partial t}\right)_f u &- 2\beta_fg_f\left(\vec\nabla_f f, \vec\nabla_f u\right)\\
&&+\beta_f h_f\left(\vec\nabla_f f,\vec\nabla_f f\right)u
\end{array} \]
The formula follows from $D_fF_\bullet(u)=-\frac{1}{2}F_f^3D_fE_\bullet(u)$.
\end{proof}

\subsection{The gradient \texorpdfstring{$\vec\nabla_fF_f$}{of F}}
Some integrations by parts will require a formula for the gradient of  $F_f$, i.e. we must differentiate it as functions of $x\in\Sigma$ when $f\in \cU$ is fixed.

% $\vec\nabla^f\left( \beta^f\right)=\left(\vec\nabla\beta\right)^f+\left(\frac{\partial \beta}{\partial t}\right)^f\nabla f$, but we will only make this substitution when necessary.\\

\begin{prop} \label{prop - gradient of F} Let $f\in \cU$. The gradient of $F_f$ satisfies:
\[g^f\left(\vec\nabla_f F_f,\vec\nabla_f f\right) =\delta(f) F_f^3\]

where $\delta(f)=\beta_f\Hess_f(f)(\vec\nabla_f f,\vec\nabla_f f)+\frac{1}{2}g_f(\vec\nabla_f f,\vec\nabla_f f)g_f\left(\vec\nabla_f \beta_f,\vec\nabla_f f\right)$ and $\Hess_f$ is the Hessian for the Riemannian metric $g_f$.
\end{prop}

\begin{rmk}
There are two possible expressions for the second order term in $\delta(f)$:
 \[ \Hess_f(f)(\vec\nabla_f f,\vec\nabla_f f) =g_f\left( \nabla^f_{\vec\nabla_f f}\vec\nabla_f f,\vec\nabla_f f\right) \]
 Here $\nabla^f$ denotes the Levi-Civita connection of the Riemannian metric $g_f$.
\end{rmk}

\begin{proof} We use the notation $E_f=F_f^{-2}=1-\beta_fg_f\left(\vec\nabla_f f,\vec\nabla_f f\right)$ once again, and focus on the gradient of $E_f$. Let $\zeta_f=g_f(\vec\nabla_f f,\vec\nabla_f f)$. Its differential is given by:
 \[\forall z\in T\Sigma\quad d\zeta_f(z)=2g_f\left(\nabla^f_z \vec\nabla_f f,\vec\nabla_f f\right)=2\Hess_f(f)\left(z,\vec\nabla_f f\right) \]
  The symmetry of $\Hess_f(f)$ yields $\vec\nabla_f \zeta_f=2\nabla^f_{\vec\nabla_f f}\vec\nabla_f f$. We can now express the gradient of $E_f=1-\beta_f\zeta_f$:
\[\vec\nabla_f E_f=-\beta_f \vec\nabla_f \zeta_f-\zeta_f\vec\nabla_f \beta_f=-2\beta_f\nabla^f_{\vec\nabla_f f}\vec\nabla_f f-g_f(\vec\nabla_f f,\vec\nabla_f f)\vec\nabla_f \beta_f\]
Since $F_f=E_f^{-\frac{1}{2}}$, we get $\vec\nabla_f F_f=-\frac{1}{2}F_f^3\vec\nabla_f E_f$, which leads to:
\[ \vec\nabla_f F_f = \beta_f F_f^3 \nabla^f_{\vec\nabla_f f}\vec\nabla_f f + \frac{1}{2}F_f^3 g_f(\vec\nabla_f f,\vec\nabla_f f) \vec\nabla_f \beta_f\]
The inner product with $\vec\nabla_ff$ gives the desired formula.
\end{proof}

\subsection{The Levi-Civita connection} The Koszul formula gives an implicit equation for the connection form $\Gamma_f$ at $f\in\cU$. For $u,v,w\in\cC^\infty(\Sigma)$:

\[G_f\left(\rule{0cm}{0.4cm}\Gamma_f(u,v),w\right) = \frac{1}{2}\left[\rule{0cm}{0.5cm}D_f\left(\rule{0cm}{0.4cm} G_\bullet(v,w)\right)(v) + D_f \left(\rule{0cm}{0.4cm}G_\bullet(u,w)\right)(v) - D_f \left(\rule{0cm}{0.4cm}G_\bullet(u,v)\right)(w) \right] \]

Our next step is to compute $D_f \left(\rule{0cm}{0.4cm}G_\bullet(u,v)\right)(w)$. Recall the expression of the metric $G$:
\[ G_f(u,v)=\int_\Sigma\beta_fF_f\dvol_{g_f} \]
The differential of $\beta_f$ is given by:
\[ D_f\beta_\bullet(u)=\left(\frac{\partial\beta}{\partial t}\right)_f u \]

The differential of $F_f$ was computed in Proposition \ref{prop - differential of F}:
 \[D_fF_\bullet (u)= \beta_f F_f^3\left[g_f\left(\vec\nabla_f f,\vec\nabla_f u\right)+\xi(f)u\right]
\]
where 
\[\xi(f)=\frac{1}{2}\left(\frac{1}{\beta}\frac{\partial \beta}{\partial t}\right)_fg_f(\vec\nabla_f f,\vec\nabla_f f)-\frac{1}{2} h_f(\vec\nabla_f f,\vec\nabla_f f)\]

Let us set  $\eta_f=\frac{1}{2}\Tr\left( (g_f)^{-1}h_{f}\right)$, so that Lemma \ref{lem - differential volume element} states:
 \[ D_f(\dvol_{g_\bullet})(u) = \eta_f u \dvol_{g_f} \]

From this we can compute the differential of the map $f\mapsto G_f(u,v)$.
\begin{align*}
D_f \left(\rule{0cm}{0.4cm}G_\bullet(u,v)\right)(w)&= \int_\Sigma F_f \left(\frac{\partial \beta}{\partial t}\right)_f uvw \dvol_{g_f} + \int_\Sigma \beta_f F_f \eta_f uvw \dvol_{g_f} \\
& \hspace{1.5cm} +\int_\Sigma \beta_f^2F_f^3 \left[g_f(\vec\nabla_f f,\vec\nabla_f w)+\xi(f)w\right]uv\dvol_{g_f}\\
&= \int_\Sigma \beta_fF_f \left[ \left(\frac{1}{\beta}\frac{\partial \beta}{\partial t}\right)_f +\eta_f +\beta_fF_f^2 \xi(f) \right]uvw \dvol_{g_f} \\
&\hspace{1.5cm}+\int_\Sigma \beta_f^2F_f^3 g_f(\vec\nabla_f f,\vec\nabla_f w)uv\dvol_{g_f}
\end{align*}
We can now find the implicit equation for the connection form $\Gamma$:
\begin{align*}
2G_f\left(\rule{0cm}{0.4cm}\Gamma_f(u,v),w\right)&= \int_\Sigma \beta_fF_f \left[ \left(\frac{1}{\beta}\frac{\partial \beta}{\partial t}\right)_f+\eta_f +\beta_fF_f^2 \xi(f) \right]uvw \dvol_{g_f} \\
&\hspace{1.5cm}+ \int_\Sigma \beta_f^2F_f^3 g_f\left(\vec\nabla_f f,\vec\nabla_f (uv)\right)w\dvol_{g_f} \\
&\hspace{1.5cm}- \int_\Sigma \beta_f^2F_f^3 g_f\left(\vec\nabla_f f,\vec\nabla_f w\right)uv\dvol_{g_f}
\end{align*}

We can use Green's formula to get rid of $\vec\nabla_f w$ in the last term:

\begin{align*}
- \int_\Sigma \beta_f^2F_f^3 g_f(\vec\nabla_f f,\vec\nabla_f w)uv\dvol_{g_f} &= \int_\Sigma \Div_f\left[uv\beta_f^2F_f^3 \vec\nabla_f f\right] w \dvol_{g_f} 
\end{align*}

%Further computations will require to develop the gradient of $\left(\beta^f\right)^2\left(F^f\right)^3$.
%\[ \boxed{\nabla F^f = \beta^f \left(F^f\right)^3 \nabla_{\nabla f}\nabla f + \frac{1}{2}\left(F^f\right)^3 \Vert \nabla f\Vert^2_{g^f} \nabla \beta^f}\]
%Note that $\nabla \beta^f=\left(\nabla\beta\right)^f+\left(\frac{\partial \beta}{\partial t}\right)^f\nabla f$, but we will only make this substitution when necessary.
%\begin{align*}\nabla \left[\left(\beta^f\right)^2\left(F^f\right)^3\right] &= 2\beta^f \left(F^f\right)^3 \nabla \beta^f \\
%&~~~+3\left(\beta^f\right)^2\left(F^f\right)^2\left[\beta^f \left(F^f\right)^3 \nabla_{\nabla f}\nabla f + \frac{1}{2}\left(F^f\right)^3 \Vert \nabla f\Vert^2_{g^f} \nabla \beta^f\right]\\
%&=\beta^f \left(F^f\right)^3 \left[2+\frac{3}{2}\beta^f\Vert \nabla f\Vert^2_{g^f}\left(F^f\right)^2 \right]\nabla\beta^f \\
%&~~~+3\left(\beta^f\right)^3\left(F^f\right)^5\nabla_{\nabla f}\nabla f
%\end{align*}
%Note that $\beta^f\Vert \nabla f\Vert^2_{g^f}\left(F^f\right)^2=\frac{\beta^f\Vert \nabla f\Vert^2_{g^f}}{1-\beta^f\Vert \nabla f\Vert^2_{g^f}}=\left(F^f\right)^2-1$, which simplifies the previous expression:
%
%\begin{align*}\nabla \left[\left(\beta^f\right)^2\left(F^f\right)^3\right] &=\frac{1}{2}\beta^f \left(F^f\right)^3 \left[3\left(F^f\right)^2 +1\right]\nabla\beta^f \\
%&~~~+3\left(\beta^f\right)^3\left(F^f\right)^5\nabla_{\nabla f}\nabla f
%\end{align*}

Here $\Div_f$ stands for the divergence with respect to the Riemannian metric $g_f$. The divergence in this last integral can be computed thanks to Proposition \ref{prop - gradient of F}.  We will  use the notation:
\begin{equation} \label{eqn - def fonction epsilon}
\e(f)=g_f\left(\vec\nabla_f \beta_f,\vec\nabla_f f\right)
\end{equation}

Write $\Delta_f$ for the Laplacian with respect to $g_f$.
\begin{align*}
\Div_f\left[uv\beta_f^2F_f^3 \vec\nabla_f f\right] &= \beta_f^2F_f^3 \left(\Delta_f f\right) uv \\
&\hspace{1cm} + \beta_f^2F_f^3 g_f\left(\vec\nabla_f f, \vec\nabla_f(uv)\right)\\
&\hspace{1cm} + 2 \beta_fF_f^3\e(f) uv \\
&\hspace{1cm}+ 3\beta_f^2F_f^5 \delta(f) uv
\end{align*}

%\begin{align*}
%- \int_\Sigma \left(\beta^f\right)^2\left(F^f\right)^3 g^f\left(\vec\nabla^f f,\vec\nabla^f w\right)uv\dvol_{g^f} &= \int_\Sigma \left(\beta^f\right)^2\left(F^f\right)^3 \left(\Delta^f f\right) uvw \dvol_{g^f} \\
%&\hspace{1cm} + \int_\Sigma \left(\beta^f\right)^2\left(F^f\right)^3 g^f\left(\vec\nabla^f f, \vec\nabla^f(uv)\right)w\dvol_{g^f}\\
%&\hspace{1cm} +2 \int_\Sigma \beta^f \left(F^f\right)^3 \e(f)uvw\dvol_{g^f}\\&
%\hspace{1cm}+3\int_\Sigma \left(\beta^f\right)^2\left(F^f\right)^5\delta(f)uvw\dvol_{g^f}
%\end{align*}
Going back to the computation of the connection form, we find:

\begin{align*}
2G_f\left(\rule{0cm}{0.4cm}\Gamma_f(u,v),w\right)&= \int_\Sigma \beta_fF_f \left[ \left(\frac{1}{\beta}\frac{\partial \beta}{\partial t}\right)_f +\eta_f \right]uvw \dvol_{g_f} \\
&\hspace{1cm}+\int_\Sigma \beta_f F_f\left[\beta_f\xi(f)+\beta_f\Delta_f f +2\e(f) \right]F_f^2 uvw\dvol_{g_f}\\
&~~~+\int_\Sigma \beta_f F_f\left[ 3\beta_f\delta(f) \right]F_f^4 uvw\dvol_{g_f}\\
&~~~+ 2\int_\Sigma \beta_f^2F_f^3 g_f\left(\vec\nabla_f f,\vec\nabla_f (uv)\right)w\dvol_{g_f} 
\end{align*}

We have proved the following result.

\begin{theorem} \label{thm - connexion de Levi-Civita}
Let $(M,g)$ be a spatially compact spacetime. The Levi-Civita connection of the $L^2$ metric on $\Cauchy(M,g)$ exists, and its expression in the graph coordinates is:
\[\boxed{\Gamma_f(u,v)=\frac{1}{2}\p(f) uv + g_f\left(\psi(f),\vec\nabla_f(uv)\right)}
\] where $\psi:\cU\to \cX(\Sigma)$ is defined by $\psi(f)=\beta_fF_f^2\vec\nabla_f f$ and  $\p:\cU\to\cC^\infty(\Sigma)$ is defined by:
\begin{align*} \p(f)&=\left(\frac{1}{\beta}\frac{\partial \beta}{\partial t}\right)_f+\eta_f\\ &+\left[\beta_f\xi(f)+\beta_f\Delta_f f +2\e(f) \right]F_f^2\\
&+ \left[ 3\beta_f\delta(f) \right]F_f^4
\end{align*}
\end{theorem}

\section{The curvature of \texorpdfstring{$\Cauchy(M,g)$}{C(M,g)}} \label{sec: curvature}

Now that we have established the existence of the Levi-Civita connection, the existence of the Riemann curvature tensor follows from the usual formula $R=D\Gamma + \Gamma\wedge \Gamma$, where the exterior derivative $D\Gamma$ is defined by:

\[ D\Gamma_f(u,v)w = D_f(\Gamma_\bullet(v,w))(u)-D_f(\Gamma_\bullet(u,w)(v) \] 
and the wedge product $\Gamma\wedge\Gamma$ is given by:
\[(\Gamma\wedge\Gamma)_f(u,v)w=\Gamma_f(u,\Gamma_f(v,w))-\Gamma_f(v,\Gamma_f(u,w)) \]

Note that since $\Sigma$ can be chosen to be any spacelike Cauchy hypersurface, i.e. any element of $\Cauchy(M,g)$, we only need to compute the curvature at $f=0$.

\vspace{\baselineskip}

 We will first focus on the exterior derivative $D\Gamma$, so we need to differentiate $f\mapsto\Gamma_f(v,w)$. It breaks down into four parts:
\begin{align}\label{eqn - differential connection form} D_f(\Gamma_\bullet(v,w))(u) &=\frac{1}{2}D_f\p(u)vw \\
&~~~+g_f\left(D_f\psi(u),\vec\nabla_f(vw)\right)\nonumber\\
&~~~+g_f\left(\psi(f),D_f\left[\vec\nabla_f(vw)\right](u)\right)\nonumber\\
&~~~+h_f\left(\psi(f),u\vec\nabla_f(vw)\right)\nonumber
\end{align}

At $f=0$, we have $\psi(0)=0$ so the two last terms vanish. This means that we only have differentiate $\psi$ and $\p$ as functions of $f$. The functions that $\psi$ is built from have already been differentiated, so this will be an easy task. Most of the work left is for the differential of $\p$.

\subsection{Differentiating \texorpdfstring{$\psi$}{Psi}} 

\begin{lemme} \label{lem - differential psi}
The differential of $\psi$ at $f=0$ is given by:
\[ \forall u\in \cC^\infty(\Sigma)\quad D_0\psi(u)=\beta_0\vec\nabla_0 u \]
\end{lemme}

\begin{proof}
 Recall the expression of $\psi$ found in Theorem \ref{thm - connexion de Levi-Civita}:
\[ \psi(f)=\beta_f F_f^2\vec\nabla_f f \]
The map $f\mapsto \vec\nabla_f f$ was differentiated in Lemma \ref{lem - differential gradient}. At $f=0$, we find $D_0\nabla (u) = \nabla_0 u$, therefore:
\[ D_0\psi(u) = \beta_0  F_0^2 \vec\nabla_0 u \]
The expression of $F_f(x)=F(x,\vec\nabla_f f(x),f(x))$ yields $F_0=1$.

\end{proof}

% \indent Since $D_0F^\bullet =0$ and $D_0\nabla(u)=\nabla u$, we get: 
%
%\[D_0\psi(u) = \beta^0\nabla u\]
%
%From this, we get the second term in the differential of $\Gamma_\bullet(v,w)$.
%\[ g^0(D_0\psi(u),\nabla(vw)) = \beta^0 g^0(\nabla u,\nabla(vw)) \]
%We can now compute the differential of $\Gamma_\bullet(v,w)$:
%\begin{align*}
%D_0(\Gamma_\bullet(v,w))(u)& =\frac{1}{2}D_0\p(u)vw\\
%&~~~ + \beta(x,0)g^0(\nabla u,\nabla(vw))
%\end{align*}
%
%
%
%The contribution to the Riemann tensor is obtained by antisymmetrizing the formula with respect to the pair $(u,v)$.
%
%\begin{align}
%R^1_0(u,v)w &= \frac{1}{2}D_0\p(u)vw-\frac{1}{2}D_0\p(v)uw \nonumber\\
%&~~~-\beta g^0(u\nabla v-v\nabla u,\nabla w)
%\label{eqn - Riemann 1}
%\end{align}

\subsection{Differentiating second order terms}
In order to differentiate $\p$, we will need to differentiate the maps $f\mapsto \Delta_f f$ and $f\mapsto \Hess_f(f)\left(\vec\nabla_f f,\vec\nabla_f f\right)$. Their expressions at an arbitrary function $f\in \cU$ are quite involved, but they are dramatically simple at $f=0$, and it turns out that they do not contribute to the curvature.

\begin{prop} \label{prop - differential second order terms}
The differential at $f=0$ of the map $\Delta:f\mapsto \Delta_f f$ satisfies:
\[\forall u,v\in\cC^\infty(\Sigma)\quad D_0\Delta(u)v-D_0\Delta(v)u = v\Delta_0u-u\Delta_0v \]
The differential at $f=0$ of the map $\theta:f\mapsto \Hess_f(f)\left(\vec\nabla_f f,\vec\nabla_f f\right)$ vanishes.

\end{prop}

\begin{proof} In order to do this, we start by  fixing a vector field $X\in\cX(\Sigma)$ and differentiating $f\mapsto \Hess_f(f)(X,X)$. We will use the following formula for the Hessian:

\begin{align}\Hess_f(f)(X,X)=X\cdot(X\cdot f)-(\nabla^f_XX)\cdot f
\label{eqn - formule hessienne}
\end{align}

Since $f\mapsto X\cdot(X\cdot f)$ is a continuous linear mapping of $\cC^\infty(\Sigma,\R)$, the problem boils down to computing the differential of $f\mapsto \nabla^f_XX$. Considering a vector field $Z\in \cX(\Sigma)$ that commutes with $X$ (which is sufficient for our purpose since we are interested in the tensor $\Hess_f(f)$), the Kozsul formula simplifies: 

\begin{align} 2g_f\left(\nabla^f_XX,Z\right)=2X\cdot g_f(X,Z)-Z\cdot g_f(X,X)
\label{eqn - Levi-Civita on Sigma}
\end{align}

The differential of the left hand term in \eqref{eqn - Levi-Civita on Sigma}, evaluated at $u\in\cC^\infty(\Sigma)$, is:
\begin{align} 2g_f\left(D_f\left(\nabla^\bullet_XX\right)(u),Z\right)+2uh_f\left(\nabla^f_XX,Z\right)
\label{eqn - Differentielle Levi-Civita lht}
\end{align}

The differential of the right hand term in \eqref{eqn - Levi-Civita on Sigma} is:
\begin{align}
2X\cdot(uh_f(X,Z))-Z\cdot(uh_f(X,X))&=2(X\cdot u)g_f(\widetilde{h_f}(X),Z)-h_f(X,X)Z\cdot u \nonumber \\
&~~~~ +u\left[2X\cdot h_f(X,Z)-Z\cdot h_f(X,X) \right]
\label{eqn - Differentielle Levi-Civita rht}
\end{align}
From \eqref{eqn - Differentielle Levi-Civita lht}$=$\eqref{eqn - Differentielle Levi-Civita rht}, we get:
\begin{align} D_f\left(\nabla^\bullet_XX\right)(u)=(X\cdot u)\widetilde{h_f}(X)-\frac{1}{2}h_f(X,X)\vec\nabla_f u + u\Lambda(f,X)
\label{eqn - Differentielle Levi-Civita formule finale}
\end{align}
where the function $\Lambda:\cU\times \cX(\Sigma)\to \cX(\Sigma)$ can be written explicitly but we do not need to.\\
We can use \eqref{eqn - Differentielle Levi-Civita formule finale} to differentiate \eqref{eqn - formule hessienne} with respect to $f$:
\begin{align*} D_f\left[\Hess_\bullet(X,X)\right](u)&=\Hess_f(u)(X,X)-g_f\left(D_f\left(\nabla^\bullet_XX\right)(u),\vec\nabla_f f\right)\\
&= \Hess_f(u)(X,X)\\
&~~~-h_f\left(X,\vec\nabla_f f\right)g_f\left(X,\vec\nabla_f u\right)+\frac{1}{2}h_f(X,X)g_f\left(\vec\nabla_f u,\vec\nabla_f f\right) \\ &~~~-g_f\left(\Lambda(f,X),\vec\nabla_f f\right)u
\end{align*}
At $f=0$, we are simply left with:
\begin{align} D_0 \Hess(X,X)(u)=\Hess_0(u)(X,X)
\label{eqn - differentielle hessienne}
\end{align}

The trace involved in the definition of the Laplacian is the trace of a bilinear form, i.e. $\Delta_f f=\Tr\left( g_f^{-1} \Hess_f \right)$, we can compute the differential of $\Delta:f\mapsto \Delta_f f$ at $f=0$.
\[ D_0\Delta(u)= -\Tr\left( g_0^{-1}h_0g_0^{-1}\Hess_0\right)u+ \Delta_0 u\]

%\[ -h^f(\nabla u,\nabla f) +\frac{1}{2}\Tr_g^f\left(h^f\right)g^f(\nabla f,\nabla u) - \lambda(f)u
%\]
%We recognize $\eta^f$:
%\[ D_f\Delta(u)= \Delta u -h^f(\nabla u,\nabla f) +\eta^fg^f(\nabla f,\nabla u) - \lambda(f)u
%\]
%We will only need an antisymmetric formula:
%\begin{align*}
%D_f\Delta(u)v-D_f\Delta(v)u &= -(u\Delta v-v\Delta u)\\
%&~~~+h^f(\nabla f,u\nabla v-v\nabla u)\\
%&~~~-\eta^f g^f(\nabla f,u\nabla v-v\nabla u)
%\end{align*}
%At $f=0$, we get:
%\[\boxed{
%D_0\Delta(u)v-D_0\Delta(v)u = v\Delta u-u\Delta v
%}\]

We now need to differentiate the map  $\theta:f\mapsto\Hess_f(f)\left(\vec\nabla_f f,\vec\nabla_f f\right)$.

\begin{align*}
D_f\theta(u)&= D_f\Hess_\bullet\left(\vec\nabla_f f,\vec\nabla_f f\right)(u)+2\Hess_f(f)\left(\vec\nabla_f f,D_f\vec\nabla (u)\right)\end{align*}

At $f=0$ we directly get $D_0\theta=0$.

%\begin{align*}
%D_f\theta(u)&= D_f\Hess(\nabla f,\nabla f)(u)+2\Hess(f)(\nabla f,D_f\nabla (u))\\
%&= \Hess(u)(\nabla f,\nabla f)-\frac{1}{2}h^f(\nabla f,\nabla f)g^f(\nabla f,\nabla u)+\widetilde\mu(f)u \\
%&~~~+2\Hess(f)(\nabla f, \nabla u-u\widetilde{h^f}(\nabla f))\\
%&= \Hess(u)(\nabla f,\nabla f) -\frac{1}{2}h^f(\nabla f,\nabla f)g^f(\nabla f,\nabla u) \\
%&~~~+2\Hess(f)(\nabla f, \nabla u) +\mu(f)u
%\end{align*}
%Here again, the functions $\widetilde\mu,\mu:\cU\to \cC^\infty(\Sigma,\R)$ do not need to be explicit, as once again we only need an antisymmetrization:
%\begin{align*}
%D_f\theta(u)v-D_f\theta(v)u &= -\left[u\Hess(v)-v\Hess(u)\right](\nabla f,\nabla f)\\
%&~~~+\frac{1}{2}h^f(\nabla f,\nabla f)g^f(\nabla f,u\nabla v-v\nabla u)\\
%&~~~-2\Hess(f)(\nabla f,u\nabla v-v\nabla u)
%\end{align*}
%The expression vanishes at  $f=0$:
%\[\boxed{
%D_0\theta(u)v-D_0\theta(v)u = 0
%}\]

\end{proof}

\subsection{Differentiating \texorpdfstring{$\p$}{Phi}}
In order to compute the curvature tensor, we will need the expression for  the differential of $\p$ at $f=0$. More precisely, the term involved in the exterior differential $D\Gamma$ has the form:
\[D_0\p(u)v-D_0\p(v)u\]
 This means that all the terms in $D_0\p(u)$ that are proportional to $u$ will disappear, leaving only the ones that involve derivatives of $u$.
 
 Recall from Theorem \ref{thm - connexion de Levi-Civita} that $\p(f)=\p_1(f)+F_f^2 \p_2(f)+F_f^4\p_3(f)$ where:
\begin{align*}
\p_1(f)&=\left(\frac{1}{\beta}\frac{\partial \beta}{\partial t}\right)_f+\eta_f\\
\p_2(f)&= \beta_f\xi(f)+\beta_f\Delta_f f +2\e(f)\\
\p_3(f)&= 3\beta_f\delta(f)
\end{align*}

It follows from Proposition \ref{prop - differential of F} that $D_0F_\bullet =0$, so we can focus on differentiating $\p_1$, $\p_2$ and $\p_3$.

\begin{lemme} \label{lem - differential phi_1}
The differential of $\p_1$ at $f=0$ satisfies:
\[ \forall u,v\in\cC^\infty(\Sigma)\quad D_0\p_1(u)v-D_0\p_1(v)u=0 \]
\end{lemme}

\begin{proof}
The value of  $\p_1(f)$ at some $x\in\Sigma$ only depends on $f(x)$, so the differential can be written  as $D_f\p_1(u)=\Phi_1(f)u$ for some $\Phi_1:\cU\to\cC^\infty(\Sigma)$, hence $D_0\p_1(u)v-D_0\p_1(v)u= \Phi_1(0)uv -\Phi_1(0)vu = 0 $.
\end{proof}

\begin{lemme} \label{lem - differential phi_2}
The differential of $\p_2$ at $f=0$ satisfies, for $u,v\in\cC^\infty(\Sigma)$:
\[  D_0\p_2(u)v-D_0\p_2(v)u= 2g_0\left(\vec\nabla_0\beta_0,v\vec\nabla_0 u-u\vec\nabla_0 v\right)+\beta_0\left[v\Delta_0 u-u\Delta_0 v\right] \]
\end{lemme}

\begin{proof}The term $\p_2=\beta_f\xi(f) + \beta_f \Delta_f f+2\e(f)$ is itself made of three parts. Recall the expression of the differential of $f\mapsto \beta_f$:
\[ D_f\beta_\bullet(u)=\left(\frac{\partial\beta}{\partial t}\right)_f u \]
 The expression of $\xi(f)$ comes from Proposition \ref{prop - differential of F}:
\[\xi(f)=\frac{1}{2}\left(\frac{1}{\beta}\frac{\partial \beta}{\partial t}\right)_fg_f\left(\vec\nabla_f f,\vec\nabla_f f\right)-\frac{1}{2} h_f(\vec\nabla_f f,\vec\nabla_f f)\]
 It differentiates as
\begin{align*}D_f\xi(u) &= \left(\frac{1}{\beta}\frac{\partial \beta}{\partial t}\right)_fg_f\left(\vec\nabla_f f,\vec\nabla_f u\right)-h_f\left(\vec\nabla_f f,\vec\nabla_f u\right)\\
&~~~ + \ell(f)u
\end{align*}
for some expression $\ell(f)$ which will disappear in the formula for $D\Gamma$. 
Indeed, at $f=0$ we find $D_0\xi(u)=\ell(0)u$ and:

\begin{equation} \label{eqn - differential phi_2 part 1}
 D_0\left[\beta_\bullet\xi\right](u)v-D_0\left[\beta_\bullet\xi\right](v)u=0 
\end{equation}

The second term $\beta_f\Delta_ff$ can be differentiated thanks to Proposition \ref{prop - differential second order terms}.

\begin{equation} \label{eqn - differential phi_2 part 2}
 D_0\left[\beta_\bullet\Delta\right](u)v-D_0\left[\beta_\bullet\Delta\right](v)u=\beta_0\left( v\Delta_0u-u\Delta_0v\right) 
\end{equation}

We finally have to differentiate $\e(f)$:

\begin{align*}
\e(f)&=g_f\left(\vec\nabla_f \beta_f,\vec\nabla_f f\right)
\end{align*}
We have not computed the differential of $f\mapsto \vec\nabla_f\beta_f$, but there is no need to since at $f=0$ this term will be multiplied by $\vec\nabla_ff=0$. The only non vanishing part of the differential at $f=0$ is obtained by differentiating $f\mapsto \vec\nabla_ff$, so the computations in Lemma \ref{lem - differential gradient} lead us to:

\begin{align*}
D_0\e(u)&= g_0\left(\vec\nabla_0\beta_0,\vec\nabla_0 u\right) 
\end{align*}
Which yields:
\begin{equation} \label{eqn - differential phi_2 part 3}
D_0\e(u)v-D_0\e(v)u= g_0\left(\vec\nabla_0 \beta_0,v\vec\nabla_0 u-u\vec\nabla_0 v\right)
\end{equation}
The contribution of $\p_2=\beta_f\xi(f) + \beta_f \Delta_f f+2\e(f)$ to $D\Gamma_0$ is obtained by combining \eqref{eqn - differential phi_2 part 1}, \eqref{eqn - differential phi_2 part 2} and \eqref{eqn - differential phi_2 part 3}.

\end{proof}

\begin{lemme} \label{lem - differential phi_3}
The differential of $\p_3$ at $f=0$ vanishes. 
\end{lemme}

\begin{proof}
Recall that $\p_3(f)=3\beta_f\delta(f)$ where  $\delta(f)=\beta_f\theta(f)+\frac{1}{2}\zeta_f\e(f)$. The map $\theta:f\mapsto \Hess_f(f)\left(\vec\nabla_f f,\vec\nabla_f f\right)$ was differentiated in Proposition \ref{prop - differential second order terms} where we found $D_0\theta=0$. We also have $\theta(0)=0$ from its expression.

  The function $\zeta_f=g_f\left(\vec\nabla_f f,\vec\nabla_f f\right)$ vanishes at $0$, and all the terms in its differential involve a product with $\nabla_ff$, so $D_0\delta=0$. But we also have $\delta(0)=0$, so $D_0\p_3=0$.
\end{proof}

We now have the full contribution of $\p$ to $D\Gamma$.

\begin{prop} \label{prop - differential phi}
The differential of $\p$ at $f=0$ satisfies, for $u,v\in\cC^\infty(\Sigma)$:
\begin{align*}
D_0\p(u)v-D_0\p(v)u&=2g_0\left(\vec\nabla_0\beta_0,v\vec\nabla_0 u-u\vec\nabla_0 v\right) +\beta_0\left[v\Delta_0 u-u\Delta_0 v\right]
\end{align*}
\end{prop}

\begin{proof}
This is a consequence of the expression \[\p(f)=\p_1(f)+F_f^2 \p_2(f)+F_f^4\p_3(f)\] and of the fact that $D_0F_\bullet =0$ (Proposition \ref{prop - differential of F}) and $F_0=1$ as well as of Lemma \ref{lem - differential phi_1}, Lemma \ref{lem - differential phi_2} and Lemma \ref{lem - differential phi_3}.
\end{proof}

\subsection{The curvature tensor of type (3,1)}
We can now compute the Riemann curvature tensor of $\Cauchy(M,g)$ explicitly.

\begin{theorem} \label{thm - (3,1) curvature tensor}
The Riemann curvature tensor $R$ of type $(3,1)$ of the $L^2$-metric on $\Cauchy(M,g)$ at $S\in \Cauchy(M,g)$  is given by, for $u,v,w\in \cC^\infty(S)$:

\begin{align*} 
R_S(u,v)w &= g\left(v\vec\nabla u-u\vec\nabla v, \vec\nabla  w\right) \\
&\hspace{1cm} + \frac{1}{2} \left[v\Delta u-u\Delta v\right] w 
\end{align*}
where the gradient $\vec\nabla$ and the Laplacian $\Delta$ are considered with respect to the induced Riemannian metric on $S$.
\end{theorem}

\begin{proof} Write $R=D\Gamma+\Gamma\wedge\Gamma$, and start with $D\Gamma$. From \eqref{eqn - differential connection form} we find:

\begin{align} \label{eqn - exterior derivative connection form part 1}
D\Gamma_0(u,v)w &= \frac{1}{2}\left(\rule{0cm}{0.4cm} D_0\p(u)v-D_0\p(v)u\right)w \\
&\hspace{1cm} + \left[g_0\left( D_0\psi(u),\vec\nabla_0v\right)-g_0\left( D_0\psi(v),\vec\nabla_0u\right) \right] w \nonumber \\
&\hspace{1cm} + g_0\left(D_0\psi(u)v-D_0\psi(v)u, \vec\nabla_0w \right)\nonumber
\end{align}

In Lemma \ref{lem - differential psi} we saw that $D_0\psi(u)=\beta_0\vec\nabla_0 u$, so the second term in \eqref{eqn - exterior derivative connection form part 1} vanishes.  The differential of $\p$ comes from Proposition \ref{prop - differential phi}, and we find:

\begin{align} \label{eqn - exterior derivative connection form part 2}
D\Gamma_0(u,v)w &= g_0\left(\vec\nabla_0\beta_0,v\vec\nabla_0 u-u\vec\nabla_0 v\right)w \nonumber\\
&\hspace{1cm} + \frac{1}{2} \beta_0\left[v\Delta_0 u-u\Delta_0 v\right] w \nonumber\\
&\hspace{1cm} +\beta_0 g_0\left(v\vec\nabla_0 u-u\vec\nabla_0 v, \vec\nabla_0w \right)\nonumber\\
&= g_0\left(v\vec\nabla_0 u-u\vec\nabla_0 v, \vec\nabla_0\left[\beta_0w\right]\right) \\
&\hspace{1cm} + \frac{1}{2} \left[v\Delta_0 u-u\Delta_0 v\right]\beta_0 w \nonumber
\end{align}

We now have to compute the wedge product $\Gamma\wedge \Gamma$, whose expression is:
\[(\Gamma\wedge\Gamma)_f(u,v)w=\Gamma_f(u,\Gamma_f(v,w))-\Gamma_f(v,\Gamma_f(u,w)) \]
Recall the expression of $\Gamma$ from Theorem \ref{thm - connexion de Levi-Civita}:
\[ \Gamma_f(u,v)=\frac{1}{2}\p(f)uv+g_f\left(\psi(f),\vec\nabla_f(uv)\right) \]
Since $\psi(0)=0$, we get:
\[ \Gamma_0(u,v)=\frac{1}{2}\p(0)uv \]
There is no need to compute $\p(0)$, since $\Gamma_0(u,\Gamma_f(v,w))=\frac{1}{4}\p(0)^2uvw$ leads directly to $(\Gamma\wedge\Gamma)_0=0$. So the expression of the Riemann tensor in coordinates follows from \eqref{eqn - exterior derivative connection form part 2}:

\begin{align*} 
R_0(u,v)w &= g\left(v\vec\nabla_0 u-u\vec\nabla_0 v, \vec\nabla_0 \left[\beta_0  w\right]\right) \\
&\hspace{1cm} + \frac{1}{2} \left[v\Delta_0 u-u\Delta_0 v\right] \beta_0 w 
\end{align*}
The coordinate free expression comes from the differential of the graph map given in Lemma \ref{lem - differential graph map}:
\[ D_0\Gr(u)=\beta_0^{\frac{1}{2}}u.\]
\end{proof}

\subsection{Sectional curvature}

We now have all the tools to prove Theorem \ref{mainthm_curvature}. From the formula for the differential of the graph map, we see that the expression in coordinates should be:

\[ \cR_0(u,v,v,u) = -\frac{1}{2}\int_\Sigma\beta_0^2 \Vert u\vec\nabla_0 v-v\vec\nabla_0 u\Vert^2_{g_0} \dvol_{g_0}. \]

\begin{proof}[Proof of Theorem \ref{mainthm_curvature}]
Since $F_0=1$, we find by using Theorem \ref{thm - (3,1) curvature tensor}:
\begin{align} \label{eqn - (4,0) curvature tensor part 1}
\cR_0(u,v,v,u)&= G_0\left(\rule{0cm}{0.4cm} R_0(u,v)v,u\right)\\
&=\int_\Sigma\beta_0g_0\left(v\vec\nabla_0 u-u\vec\nabla_0 v, \vec\nabla_0\left[\beta_0v\right]\right)u\dvol_{g_0}\nonumber \\
&\hspace{1cm} + \frac{1}{2}\int_\Sigma\beta_0^2 \left[v\Delta_0 u-u\Delta_0 v\right]uv \dvol_{g_0} \nonumber
\end{align}

Note that $v\Delta_0 u-u\Delta_0 v=\Div_0\left(v\vec\nabla_0 u-u\vec\nabla_0 v\right)$, so we can apply Green's formula to the second term:

\begin{align*}
\int_\Sigma\beta_0^2 \left[v\Delta_0 u-u\Delta_0 v\right]uv \dvol_{g_0} &= -\int_\Sigma g_0\left(v\vec\nabla_0 u-u\vec\nabla_0 v, \vec\nabla_0\left[\beta_0^2uv\right]\right) \dvol_{g_0}
\end{align*}
Let us write $\beta_0^2uv$ as $\left(\beta_0u\right)\left(\beta_0v\right)$ when developing its gradient:

\begin{align}\label{eqn - IPP2}
\int_\Sigma\beta_0^2 \left[v\Delta_0 u-u\Delta_0 v\right]uv \dvol_{g_0} &= -\int_\Sigma\beta_0 g_0\left(v\vec\nabla_0 u-u\vec\nabla_0 v, \vec\nabla_0\left[\beta_0u\right]\right)v \dvol_{g_0}\nonumber\\
&\hspace{0.4cm} - \int_\Sigma\beta_0 g_0\left(v\vec\nabla_0 u-u\vec\nabla_0 v, \vec\nabla_0\left[\beta_0v\right]\right)u \dvol_{g_0} 
\end{align}

We can reinject \eqref{eqn - IPP2} back into \eqref{eqn - (4,0) curvature tensor part 1}.
\begin{align*}
\cR_0(u,v,v,u)&= -\frac{1}{2}\int_\Sigma\beta_0g_0\left(v\vec\nabla_0 u-u\vec\nabla_0 v,v \vec\nabla_0\left[\beta_0u\right]-u \vec\nabla_0\left[\beta_0v\right]\right)u\dvol_{g_0}\\
&= -\frac{1}{2}\int_\Sigma \beta_0^2 g_0\left(v\vec\nabla_0 u-u\vec\nabla_0 v,v\vec\nabla_0 u-u\vec\nabla_0 v\right)\dvol_{g_0}
\end{align*}

\end{proof}

\vspace{0.5cm}
 {\footnotesize \emph{E-mail address:}  daniel.monclair@universite-paris-saclay.fr
 
  \textsc{Université Paris-Saclay, CNRS, Laboratoire de mathématiques d’Orsay, 91405, Orsay, France}}

\end{document}